\documentclass{amsart}

\newtheorem{theorem}{Theorem}[section]
\newtheorem{lemma}[theorem]{Lemma}
\newtheorem{proposition}{Proposition}[section]
\newtheorem{corollary}{Corollary}[section]

\theoremstyle{definition}

\newtheorem*{problem}{Problem}

\theoremstyle{remark}

\numberwithin{equation}{section}

\newcommand{\abs}[1]{\lvert#1\rvert}

\begin{document}

\title{Non-separable tree-like Banach spaces and Rosenthal's $\ell_1$-theorem}

\author{Costas Poulios}
\address{Department of Mathematics, University of Athens, 15784, Athens, Greece}
\email{k-poulios@math.uoa.gr}

\subjclass[2010]{Primary 46B25, 46B26.}

\date{}


\keywords{Non-separable tree-like Banach spaces, Rosenthal's
$\ell_1$-theorem, uncountable cardinal.}


\begin{abstract}
We introduce and investigate a class of non-separable tree-like
Banach spaces. As a consequence, we prove that we can not achieve a
satisfactory extension of Rosenthal's $\ell_1$-theorem to spaces of
the type $\ell_1(\kappa)$, for $\kappa$ an uncountable cardinal.
\end{abstract}

\maketitle

\section{Introduction}\label{sec.Intr.}
Rosenthal's $\ell_1$-theorem \cite{Rosenthal} is one of the most
remarkable results in Banach space geometry. It provides a
fundamental criterion for the embedding of $\ell_1$ into Banach
spaces.

\begin{theorem}[Rosenthal's $\ell_1$-theorem]\label{th.Rosenthal}
Let $(x_n)$ be a bounded sequence in the Banach space $X$ and
suppose that $(x_n)$ has no weakly Cauchy subsequence. Then $(x_n)$
contains a subsequence equivalent to the usual $\ell_1$-basis.
\end{theorem}

A satisfactory extension of Theorem \ref{th.Rosenthal} to spaces of
the type $\ell_1(\kappa)$, for $\kappa$ an uncountable cardinal,
would be desirable, since it would provide a useful criterion for
the embedding of $\ell_1(\kappa)$ into Banach spaces. Naturally,
therefore, R. G. Haydon \cite{Haydon2}  posed the following problem:
Let $\kappa$ be an uncountable cardinal. Suppose that $X$ is a
Banach space, $A$ is a bounded subset of $X$ whose cardinality is
equal to $\kappa$ and such that $A$ does not contain any weakly
Cauchy sequence. Can we deduce that $A$ has a subset equivalent to
the usual $\ell_1(\kappa)$-basis?

Before the question was posed, Haydon \cite{Haydon1} had already
presented a counterexample for the case where the cardinal $\kappa$
is equal to $\omega_1$. A completely different counterexample for
the case of $\omega_1$ had also been obtained by J. Hagler
\cite{Hagler}. Finally, the complete solution to the aforementioned
problem was given by C. Gryllakis \cite{Gryll} who proved that the
answer is always negative with only one exception, namely when both
$\kappa$ and $cf(\kappa)$ are strong limit cardinals.

In this paper, we first introduce for any infinite cardinal $\kappa$
a tree-like Banach space $X_\kappa$. Our construction is motivated
by the well-known James Tree space ($JT$) \cite{James} and Hagler
Tree space ($HT$) \cite{Hagler}. We also study in detail various
properties of the space $X_\kappa$ and we mostly focus on a family
of continuous functionals defined on $X_\kappa$. As a consequence of
our investigation we give a very simple answer to Haydon's problem.

Closing this introductory section, we recall some definitions for
the sake of completeness. A sequence $(x_n)_{n\in\mathbb{N}}$ in a
Banach space $X$ is \emph{weakly Cauchy} if the scalar sequence
$(f(x_n))_{n\in\mathbb{N}}$ converges for every $f$ in $X^\ast$. A
subset $A\subset X$ with cardinality $\kappa$ is \emph{equivalent to
the usual $\ell_1(\kappa)$-basis} if there are constants $C_1,C_2>0$
such that $C_1\sum_{i=1}^n\abs{a_i}\le \|\sum_{i=1}^n a_i x_i\| \le
C_2 \sum_{i=1}^n \abs{a_i}$, for any $n\in\mathbb{N}$, any
$x_1,x_2,\ldots,x_n\in A$ and any scalars $a_1,\ldots,a_n$.

Finally, we should mention that this is not the first time
non-separable tree-like Banach spaces have been defined (e.g. see
\cite{Brack} and \cite{Hagler-Odell}).

\section{The basic construction} \label{sec.Construction}
Suppose that $\kappa$ is an infinite cardinal. Then we set
\begin{align*}
    \Gamma=\{0,1\}^\kappa &= \Big\{a:\{\xi<\kappa\} \to \{0,1\}\Big\}  = \Big\{ (a_\xi)_{\xi<\kappa} \mid a_\xi=0~ \text{or}~ 1\Big\}
    \\
    \mathcal{D} = \{0,1\}^{<\kappa} &= \bigcup\Big\{\{0,1\}^\eta \mid \text{Ord}(\eta),
    \eta<\kappa\Big\} \\
    &= \Big\{ (a_\xi)_{\xi<\eta} \mid \eta ~ \text{is an ordinal},~ \eta<\kappa, ~ a_\xi=0~\text{or}~1\Big\}.
\end{align*}
The set $\mathcal{D}$ is called the \emph{(standard) tree}. The
elements $s\in\mathcal{D}$ are called \emph{nodes}. The elements of
the set $\Gamma=\{0,1\}^\kappa$ are called \emph{branches}.

If $s$ is a node and $s\in\{0,1\}^\eta$, we say that $s$ is on the
\emph{$\eta$-th level of $\mathcal{D}$}. We denote the level of $s$
by $lev(s)$. The \emph{initial segment partial ordering} on
$\mathcal{D}$, denoted by $\le$, is defined as follows: if
$s=(a_\xi)_{\xi<\eta_1}$ and $s'=(b_\xi)_{\xi<\eta_2}$ belong to
$\mathcal{D}$ then $s\le s'$ if and only if $\eta_1\le \eta_2$ and
$a_\xi=b_\xi$ for any $\xi<\eta_1$. We also write $s<s'$ if $s\le
s'$ and $s\neq s'$. By $s\bot s'$ we mean that $s,s'$ are
\emph{incomparable}, that is neither $s\le s'$ nor $s'\le s$. If
$s\le s'$ we say $s'$ is a \emph{follower} of $s$. Further, the
nodes $s\cup\{0\}$ and $s\cup\{1\}$ are called the \emph{successors}
of $s$, that is we reserve the word successor as meaning immediate
follower. However, we observe that a node does not need to have an
\emph{immediate predecessor}.

A subset $T$ of $\mathcal{D}$ is called a \emph{subtree} if it is
order isomorphic to $\{0,1\}^{<\lambda}$ for some cardinal
$\lambda\le \kappa$. In this paper, we only use countable subtrees
of $\mathcal{D}$, that is subtrees order isomorphic to
$\{0,1\}^{<\aleph_0}$. In the case $T$ is countable, we enumerate
its elements as $T=\{t_1,t_2,t_3,\ldots\}$ where $t_1$ is the
minimum element of $T$ and for each $m\in\mathbb{N}$,
$t_{2m},t_{2m+1}$ are the successors (on the tree $T$) of $t_m$.

A linearly ordered subset $\mathcal{I}$ of $\mathcal{D}$ is called a
\emph{segment} if for every $s<t<s'$, $t$ is contained in
$\mathcal{I}$ provided that $s,s'$ belong to $\mathcal{I}$. Consider
now a non-empty segment $\mathcal{I}$. Let $\eta_1$ be the least
ordinal such that there exists a node $s\in\mathcal{D}$ with
$lev(s)=\eta_1$ and $s\in\mathcal{I}$. Suppose further that there
are an ordinal $\eta$ and a node $s'$ on the $\eta$-th level so that
$s\le s'$ for every $s\in\mathcal{I}$. Let $\eta_2$ be the least
ordinal satisfying this property. Then we say that $\mathcal{I}$ is
an $\eta_1$-$\eta_2$ segment. A segment is called \emph{initial} if
$\eta_1=0$, that is $\emptyset\in\mathcal{I}$.

We next define admissible families of segments in the sense of
Hagler \cite{Hagler}. Suppose that $\{\mathcal{I}_j\}_{j=1}^r$ is a
finite family of segments. This family is called \emph{admissible}
if the following conditions are satisfied:
\begin{enumerate}
    \item there exist ordinals $\eta_1<\eta_2$ such that
    $\mathcal{I}_j$ is an $\eta_1$-$\eta_2$ segment for each
    $j=1,\ldots,r$;
    \item $\mathcal{I}_i\cap\mathcal{I}_j=\emptyset$ provided that
    $i\neq j$.
\end{enumerate}

Consider now the vector space $c_{00}(\mathcal{D})$ of finitely
supported functions $x:\mathcal{D}\to\mathbb{R}$. For any segment
$\mathcal{I}$ of $\mathcal{D}$, we set
$\mathcal{I}^\ast:c_{00}(\mathcal{D})\to\mathbb{R}$ with
$\mathcal{I}^\ast(x)=\sum_{s\in\mathcal{I}}x(s)$. Then, for any
$x\in c_{00}(\mathcal{D})$, we define the norm
$$\|x\|=\sup\Big[\sum_{j=1}^r
\abs{\mathcal{I}_j^\ast(x)}^2\Big]^{1/2}$$ where the supremum is
taken over all finite, admissible families
$\{\mathcal{I}_j\}_{j=1}^r$ of segments. The space $X_\kappa$ is the
completion of the normed space $(c_{00}(\mathcal{D}),\|.\|)$ we have
just defined.

For every node $s\in\mathcal{D}$, we define
$e_s:\mathcal{D}\to\mathbb{R}$ with $e_s(t)=1$ if $t=s$ and
$e_s(t)=0$ otherwise. Clearly, $\|e_s\|=1$ for any
$s\in\mathcal{D}$.

We come now to the final definition. Suppose that $\{s_i\mid i\in
I\}$ is a family of nodes of the tree $\mathcal{D}$. This family is
called \emph{strongly incomparable} (see \cite{Hagler}) if the
following hold:
\begin{enumerate}
    \item $s_i\bot s_j$ provided that $i\neq j$;
    \item if $\{S_1,\ldots,S_r\}$ is any admissible family of
    segments, then at most two nodes of the $s_i$'s, $i\in I$, are
    contained in $S_1\cup\ldots\cup S_r$.
\end{enumerate}
There is a standard way for constructing strongly incomparable
families of nodes. Suppose that $(s_\xi)_{\xi<\eta}$ is a set of
nodes, where $\eta<\kappa$, such that $s_0<s_1<\ldots$. For any
ordinal $\xi<\eta$, let $t_\xi$ be the successor of $s_\xi$ with
$t_\xi\bot s_{\xi+1}$. Then, the family $\{t_\xi \mid \xi<\eta\}$ is
strongly incomparable.

Concerning strongly incomparable sets of nodes, we quote the
following proposition whose proof is straightforward.

\begin{proposition}\label{prop.c0}
Suppose that $\{s_i\mid i\in I\}$ is a strongly incomparable set of
nodes on the tree $\mathcal{D}$. Then the family $\{e_{s_i} \mid
i\in I\}$ is equivalent to the usual basis of $c_0(I)$. More
precisely, for any $n\in\mathbb{N}$, any $i_1,\ldots,i_n \in I$ and
any scalars $a_1,\ldots,a_n$, we have
$$\max_{1\le k\le n}\abs{a_k} \le \Big\|\sum_{k=1}^n a_k e_{s_{i_k}}
\Big\| \le \sqrt{2} \max_{1\le k\le n}\abs{a_k}.$$
\end{proposition}

\section{The main results}\label{sec.Main Results}
Suppose that $B=(a_\xi)_{\xi<\kappa}\in\Gamma$ is any branch. Then
$B$ can be naturally identified with a maximal segment of
$\mathcal{D}$, namely $B=\{s_0<s_1<\ldots<s_\eta<\ldots\}$ where
$s_0=\emptyset$ and $s_\eta=(a_\xi)_{\xi<\eta}$ for any ordinal
$\eta<\kappa$. In Section \ref{sec.Construction}, we defined the
linear functional $B^\ast:c_{00}(\mathcal{D})\to\mathbb{R}$ by
setting $B^\ast(x)=\sum_{s\in B}x(s)$. Clearly, $\|B^\ast\|=1$. This
functional can be extended to a bounded functional on $X_\kappa$,
having the same norm and which is denoted again by $B^\ast$. Let
also $\Gamma^\ast$ denote the set which contains the functionals
$B^\ast$ defined above. Then $\Gamma^\ast$ is a bounded subset of
$X_\kappa^\ast$ whose cardinality is equal to $2^\kappa$.

This section is devoted to the study of the family $\Gamma^\ast$.
Towards this direction, we first prove the following.

\begin{theorem}\label{th.no weakly cauchy}
Suppose that $(B_n)_{n\in\mathbb{N}}$ is a sequence of branches such
that $B_n\neq B_m$ for $n\neq m$. Then $(B_n^\ast)_{n\in\mathbb{N}}$
contains a subsequence equivalent to the usual $\ell_1$-basis.
\end{theorem}

\begin{proof}
Consider the set $\mathcal{A}$ consisting of all ordinals
$\eta<\kappa$ which satisfy the following: there are nodes
$\varphi\neq t$ with $lev(\varphi)=lev(t)=\eta$ and there are
positive integers $m_1\neq m_2$ such that $\varphi\in B_{m_1}$,
$t\in B_{m_2}$. Clearly $\mathcal{A}$ is a non-empty set, therefore
we can consider its least element, say $\eta$. Then $\eta$ can not
be a limit ordinal. Indeed, let $\varphi=(a_\xi)_{\xi<\eta}$ and
$t=(b_\xi)_{\xi<\eta}$ be as above. Since $\varphi\neq t$, there
exists $\eta_1<\eta$ with $a_{\eta_1}\neq b_{\eta_1}$. We set
$\tilde{\varphi}=(a_\xi)_{\xi<\eta_1+1}$ and
$\tilde{t}=(b_\xi)_{\xi<\eta_1+1}$. Now we observe that
$\tilde{\varphi}\neq \tilde{t}$, these nodes are placed on the same
level and $\tilde{\varphi}\le \varphi$, $\tilde{t}\le t$. Hence,
$\tilde{\varphi}\in B_{m_1}$, $\tilde{t}\in B_{m_2}$. By the
minimality of $\eta$, we conclude that $\eta=\eta_1+1$.

Furthermore, the minimality of $\eta$ also implies that there exists
a node $s_1$ on the level $\eta_1$, so that $s_1\in B_m$, for every
$m\in\mathbb{N}$, and the nodes $\varphi$, $t$ on the level
$\eta=\eta_1+1$ are precisely the successors of $s_1$. Now, we set
$\varphi_1=\varphi$ and $t_1=t$. We may assume that there are
infinitely many terms of the sequence $(B_m)_{m\in\mathbb{N}}$ which
pass through the node $\varphi_1$. Then we choose a branch $B_{l_1}$
passing through the node $t_1$ (clearly such a branch does exist).
$B_{l_1}$ is just the first term of the desired subsequence.

We next set $N_1=\{m\in\mathbb{N} \mid m>l_1 ~\text{and}~
\varphi_1\in B_m\}$. Then $N_1$ is an infinite subset of
$\mathbb{N}$. Repeating the previous argument to the branches
$(B_m)_{m\in N_1}$, we find an ordinal $\eta_2>\eta_1+1$ and a node
$s_2$ on the $\eta_2$-th level with successors $\varphi_2$ and
$t_2$, such that
\begin{itemize}
    \item all branches $B_m$, $m\in N_1$, pass through the node
    $s_2$;
    \item infinitely many branches of the sequence $(B_m)_{m\in
    N_1}$ pass through $\varphi_2$ and the set $\{m\in N_1 \mid t_2\in
    B_m\}$ is non-empty.
\end{itemize}
We also choose a branch $B_{l_2}$ so that $t_2\in B_{l_2}$.

Continue in the obvious manner. We inductively construct a sequence
$s_1<s_2<\ldots$ of nodes of $\mathcal{D}$, with the successors of
$s_i$ denoted by $\varphi_i$ and $t_i$, and a sequence
$l_1<l_2<\ldots$ of positive integers such that the following hold:
\begin{enumerate}
    \item $s_1<\varphi_1\le s_2<\varphi_2\le s_3\ldots$;
    \item $s_i\in B_{l_j}$ for any $j\ge i$, however the branches
    $B_{l_j}$, $j>i$, pass through the node $\varphi_i$ while the
    branch $B_{l_i}$ passes through the node $t_i$.
\end{enumerate}

We prove now that the sequence $(B_{l_m}^\ast)_{m\in\mathbb{N}}$ is
equivalent to the usual $\ell_1$-basis. Let $M\in\mathbb{N}$ and
$a_1,\ldots,a_M\in\mathbb{R}$ be given. We set $x=\sum_{i=1}^M
sgn(a_i) e_{t_i}$. Condition $(1)$ of the above construction implies
that the sequence $(t_i)$ is strongly incomparable. Hence by
Proposition \ref{prop.c0}, we have $\|x\|=\sqrt{2}$. Furthermore,
condition $(2)$ implies that $t_i\in B_{l_i}\setminus
\cup\{B_{l_j}\mid j\neq i\}$, thus $B_{l_j}(e_{t_i})=\delta_{ij}$.
Therefore:
$$
    \Big\|\sum_{i=1}^M a_i B_{l_i}^\ast\Big\| \geq \frac{1}{\|x\|}
    \Big|\sum_{i=1}^M a_i B_{l_i}^\ast(x)\Big| = \frac{1}{\sqrt{2}}
    \Big|\sum_{i=1}^M a_i sgn(a_i)\Big|\\
    =\frac{1}{\sqrt{2}} \sum_{i=1}^M \abs{a_i}.
$$
Clearly, we have $\|\sum_{i=1}^M a_i B_{l_i}^\ast\|\le \sum_{i=1}^M
\abs{a_i}$ and the proof is complete.
\end{proof}

\begin{corollary}\label{cor.no weakly cauchy}
The set $\Gamma^\ast$ contains no weakly Cauchy sequence.
\end{corollary}

We pass now to the second result concerning the set of functionals
$\{B^\ast \mid B\in\Gamma\}$.

\begin{theorem}\label{th.basiko}
There exists no subset of $\Gamma^\ast$ which is equivalent to the
usual $\ell_1(\kappa^+)$-basis.
\end{theorem}

For the proof of the above theorem we need to establish some lemmas.
Before proceeding, let us introduce some notation. First of all, if
$A$ is any set, then $|A|$ denotes the cardinality of $A$. Suppose
now that $\Delta\subseteq \Gamma$ is a set of branches. For any node
$s\in\mathcal{D}$, we denote $\Delta_s$ the set of all branches
$B\in \Delta$ passing through $s$, that is
$\Delta_s=\{B\in\Delta\mid s\in B\}$. We also set
$\Delta_s^c=\Delta\setminus \Delta_s=\{B\in\Delta\mid s \notin B\}$.

\begin{lemma}\label{lemma1}
Let $\Delta\subseteq \Gamma$ be a set of branches with
$|\Delta|=\kappa^+$. Then there exists a node $s\in\mathcal{D}$ such
that $|\Delta_{s\cup\{0\}}|=|\Delta_{s\cup\{1\}}|=\kappa^+$
\end{lemma}

\begin{proof}
Assume that the assertion is not true. Then for every node
$s\in\mathcal{D}$ there is a successor $s\cup\{\epsilon\}$ of $s$,
where $\epsilon=0$ or $1$, such that
$|\Delta_{s\cup\{\epsilon\}}|<\kappa^+$. With this assumption and
using transfinite induction we construct a branch
$B=\{s_\eta\}_{\eta<\kappa}=\{s_0<s_1<\ldots\}$ with the property
that $|\Delta_{s_\eta}|=\kappa^+$ for any $\eta<\kappa$.

We start with $s_0=\emptyset$. Clearly,
$|\Delta_\emptyset|=|\Delta|=\kappa^+$. Suppose now that $\eta$ is
an ordinal, $\eta<\kappa$, and we have defined the nodes
$\{s_\xi\}_{\xi<\eta}$ with $lev(s_\xi)=\xi$ and
$|\Delta_{s_\xi}|=\kappa^+$ for any $\xi<\eta$.

If $\eta=\eta_0+1$, then by the inductive hypothesis we have
$|\Delta_{s_{\eta_0}}|=\kappa^+$. Clearly,
$\Delta_{s_{\eta_0}}=\Delta_{s_{\eta_0}\cup\{0\}}\cup\Delta_{s_{\eta_0}\cup\{1\}}$.
Therefore, there exists a successor $s_{\eta_0}\cup\{\epsilon\}$
(where $\epsilon=0$ or $1$) of $s_{\eta_0}$ such that
$|\Delta_{s_{\eta_0}\cup\{\epsilon\}}|=\kappa^+$. Let
$s_\eta=s_{\eta_0}\cup\{\epsilon\}$.

If $\eta$ is a limit ordinal, we set $s_\eta=\cup_{\xi<\eta}s_\xi$.
Then $s_\eta$ is a node on the $\eta$-th level of $\mathcal{D}$. It
remains to show that $|\Delta_{s_\eta}|=\kappa^+$. Since,
$\Delta=\Delta_{s_\eta}\cup \Delta_{s_\eta}^c$, it suffices to prove
that $|\Delta_{s_\eta}^c|\le \kappa$.

Let us consider a branch $B$ belonging to $\Delta_{s_\eta}^c$, that
is $s_\eta\notin B$. We also denote $S$ the initial segment
$\{s_\xi\}_{\xi\le \eta}$. We consider now the set $\mathcal{A}$
containing all ordinals $\xi\le \eta$ such that at the $\xi$-th
level of $\mathcal{D}$, the segments $B$ and $S$ do not pass through
the same node. The set $\mathcal{A}$ is non-empty as
$\eta\in\mathcal{A}$. Therefore $\mathcal{A}$ has a minimum element,
say $\xi_0$. The minimality of $\xi_0$ implies that $\xi_0$ can not
be a limit ordinal. Hence $\xi_0=\xi+1$. Further, it follows by the
minimality of $\xi_0$ that at the level $\xi$, we have $s_\xi\in B$
and $s_\xi\in S$, while at the level $\xi+1$, $s_{\xi+1}\in S$ and
$s_{\xi+1}\notin B$. Consequently,
\begin{align*}
    \Delta_{s_\eta}^c & = \cup_{\xi<\eta} \{B\in\Delta \mid
s_\xi \in B ~ \text{and}~ s_{\xi+1}\notin B\}\\
&= \cup_{\xi<\eta} (\Delta_{s_\xi} \cap \Delta_{s_{\xi+1}}^c).
\end{align*}
Observe that $s_{\xi+1}$ is a successor of $s_\xi$,
$|\Delta_{s_\xi}|=|\Delta_{s_{\xi+1}}|=\kappa^+$ and $\Delta_{s_\xi}
\cap \Delta_{s_{\xi+1}}^c$ consists of all branches $B\in\Delta$
which pass through the other successor of $s_\xi$. By our assumption
in the beginning of the proof, we have $|\Delta_{s_\xi} \cap
\Delta_{s_{\xi+1}}^c|\le \kappa$ and therefore
$|\Delta_{s_\eta}^c|\le \sum_{\xi<\eta}\kappa=\kappa$.

Therefore a branch $B=\{s_\eta\}_{\eta<\kappa}$ has been constructed
with the property $|\Delta_{s_\eta}|=\kappa^+$ for any
$\eta<\kappa$. To complete the proof of the lemma, we only need to
repeat our last argument. Consider a branch $\tilde{B}\in\Delta$
with $\tilde{B}\neq B$. Let $\xi_0$ be the minimum ordinal such that
at the $\xi_0$-th level the branches $\tilde{B},B$ do not pass
through the same node. The minimality of $\xi_0$ implies that
$\xi_0=\xi+1$, $s_\xi\in\tilde{B}$ and $s_{\xi+1}\notin \tilde{B}$.
Therefore
$$\Delta\subseteq \{B\}\cup
\Big(\cup_{\xi<\kappa}(\Delta_{s_\xi}\cap\Delta_{s_{\xi+1}}^c)\Big).$$
Since $|\Delta_{s_\xi}\cap\Delta_{s_{\xi+1}}^c|\le \kappa$, it
follows that $|\Delta|\le\kappa$ and we have reached a
contradiction.
\end{proof}

\begin{lemma}\label{lemma2}
Let $\Delta\subset\Gamma$ be a set of branches with
$|\Delta|=\kappa^+$. Then there exists a countable subtree $T$ of
$\mathcal{D}$, $T=\{t_1,t_2,t_3,\ldots\}$, such that the following
hold:
\begin{enumerate}
    \item $|\Delta_{t_m}|=\kappa^+$ for any node $t_m\in T$;
    \item for any node $t_m\in T$ there exists a node
    $s_m\in\mathcal{D}$, so that $t_m\le s_m$ and $t_{2m},t_{2m+1}$
    are the successors of $s_m$ (that is, when we look at the tree
    $\mathcal{D}$, then the successors of $t_m$ still remain the
    successors of some node $s_m\in\mathcal{D}$).
\end{enumerate}
\end{lemma}

\begin{proof}
Let $t_1=\emptyset$. By Lemma \ref{lemma1}, there exists a node
$s_1\in\mathcal{D}$, with $t_1\le s_1$ such that
$|\Delta_{s_1\cup\{0\}}|=|\Delta_{s_1\cup\{1\}}|=\kappa^+$. We set
$t_2=s_1\cup\{0\}$ and $t_3=s_1\cup\{1\}$. Then $t_2,t_3$ are the
successors of $t_1$ in $T$ and they are the successors of $s_1$ when
we look at the tree $\mathcal{D}$.

Applying Lemma \ref{lemma1} to the family
$\Delta_{s_1\cup\{0\}}=\Delta_{t_2}$ we find a node
$s_2\in\mathcal{D}$, with $t_2\le s_2$, such that
$|\Delta_{s_2\cup\{0\}}|=|\Delta_{s_2\cup\{1\}}|=\kappa^+$. Then the
successors of $t_2$ in $T$ are the nodes $t_4=s_2\cup\{0\}$ and
$t_5=s_2\cup\{1\}$. We continue in the obvious manner.
\end{proof}

\begin{proof}[Proof of Theorem \ref{th.basiko}]
Assume that $\Delta\subseteq \Gamma$ is a set of branches with
$|\Delta|=\kappa^+$ and $\Delta^\ast=\{B^\ast \mid B\in\Delta\}$ is
equivalent to the usual $\ell_1(\kappa^+)$-basis. Then there exists
a constant $\delta>0$ such that for any $n\in\mathbb{N}$, any
$B_1,\ldots,B_n\in\Delta$ and any scalars $a_1,\ldots,a_n$,
$$\delta\sum_{i=1}^n\abs{a_i} \le \Big\|\sum_{i=1}^n a_iB_i^\ast
\Big\| \le \sum_{i=1}^n\abs{a_i} .$$

Let $T$ be the countable subtree of $\mathcal{D}$ given by Lemma
\ref{lemma2} and let $n\in \mathbb{N}$ be any positive integer. Then
we choose branches $B_1,\ldots,B_n$ and $B_{n+1},\ldots,B_{2n}$
belonging to $\Delta$ as follows. We work at the $n$-th level of $T$
which consists of the nodes $t_{2^n}, t_{2^n+1},t_{2^n+2},\ldots,
t_{2^{n+1}-1}$. If we consider the pair $t_{2^n}, t_{2^n+1}$, the
construction of the tree $T$ implies that these nodes are the
successors of some node of the tree $\mathcal{D}$. Therefore they
belong to the same level of $\mathcal{D}$, say the level $\xi_1$.
Similarly the nodes $t_{2^n+2}, t_{2^n+3}$ are placed on the same
level of $\mathcal{D}$, say $\xi_2$, and so on. Finally, let
$\xi_{2^{n-1}}=lev(t_{2^{n+1}-2})=lev(t_{2^{n+1}-1})$. We may
assume, without loss of generality, that $\xi_1=\max\{\xi_k \mid
1\le k\le 2^{n-1}\}$. Then we choose branches $B_1$ and $B_{n+1}$ of
the family $\Delta$ such that $B_1$ passes through $t_{2^n}$ and
$B_{n+1}$ passes through $t_{2^n+1}$ (such branches exist by Lemma
\ref{lemma2}). If $\psi_1$ denotes the immediate predecessor (on the
tree $\mathcal{D}$) of the nodes $t_{2^n}, t_{2^n+1}$, then the
branches $B_1,B_{n+1}$ coincide up to the level of $\psi_1$ and they
separate each other at the next level.

The nodes $t_{2^n}, t_{2^n+1}$ are followers of the node $t_2$ in
the tree $T$. We now forget the followers of $t_2$ and we repeat the
previous procedure to the nodes belonging to the $n$-th level of $T$
which are followers of $t_3$. That is, we detect the pair, say
$t_{2^n+2k}, t_{2^n+2k+1}$, which is placed on the greatest level of
$\mathcal{D}$ (if this is not unique, we simply choose one). Then we
choose branches $B_2,B_{n+2}$ belonging to $\Delta$ such that $B_2$
passes through the left-hand node of the pair, i.e. the node
$t_{2^n+2k}$, and $B_{n+2}$ passes through the right-hand node
$t_{2^n+2k+1}$. Let $\psi_2$ denote the immediate predecessor of
$t_{2^n+2k}, t_{2^n+2k+1}$ on the tree $\mathcal{D}$. Then
$lev(\psi_1)\ge lev(\psi_2)$. The branches $B_2,B_{n+2}$ coincide up
to the level of $\psi_2$. We also notice that the branches $B_1,B_2$
separate each other before the level of $t_2,t_3$ and this happens
for the branches $B_{n+1},B_{n+2}$. The nodes $t_{2^n+2k},
t_{2^n+2k+1}$ are followers either of $t_6$ or $t_7$. If $t_6$ is a
predecessor of $t_{2^n+2k}, t_{2^n+2k+1}$, then we forget the
followers of $t_6$ and we continue with the nodes belonging to the
$n$-th level of $T$ which are followers of $t_7$.

After $n-1$ iterated applications of the previous argument, we find
branches $B_1,\ldots,B_{n-1}$ and $B_{n+1},\ldots, B_{2n-1}$ of the
family $\Delta$ and nodes $\psi_1,\ldots,\psi_{n-1}$ of
$\mathcal{D}$. At this stage only one pair of nodes on the $n$-th
level of $T$ has been left. Let $\psi_n$ be the immediate
predecessor on $\mathcal{D}$ of these nodes. We choose
$B_n,B_{2n}\in\Delta$ such that $B_n$ passes through the left-hand
node and $B_{2n}$ passes through the right-hand node.

Now we observe that the branches $B_1,\ldots,B_n$ are pairwise
disjoint below the level of $\psi_n$ and this is also true for the
branches $B_{n+1},\ldots, B_{2n}$. Therefore, if
$\eta_1=lev(\psi_n)$ and $\eta_2=lev(\psi_1)$, then the following
hold.

\begin{enumerate}
    \item All segments $B_i\cap\{s\mid lev(s)\ge \eta_2+1\}$,
    $i=1,2,\ldots,2n$, are pairwise disjoint.
    \item The segments $B_i\cap\{s\mid \eta_1+1\le lev(s)\le
    \eta_2\}$ for $i=1,2,\ldots,n$ are pairwise disjoint. Hence they
    are admissible $(\eta_1+1)$-$(\eta_2+1)$ segments. Similarly, $B_i\cap\{s\mid \eta_1+1\le lev(s)\le
    \eta_2\}$, $i=n+1,\ldots,2n$, form an admissible family.
    \item $B_i\cap\{s\mid lev(s)\le \eta_1\}=B_{n+i}\cap\{s\mid lev(s)\le
    \eta_1\}$ for any $i=1,2,\ldots,n$. Let us also denote $S_i=B_i\cap\{s\mid lev(s)\le
    \eta_1\}$.
\end{enumerate}

After the choice of $(B_i)_{i=1}^{2n}$ has been completed, our next
purpose is to estimate the norm of the functional
$\sum_{i=1}^{2n}a_iB_i^\ast$ for any scalars $a_1,\ldots,a_{2n}$ and
to contradict the assumption that $\Delta^\ast$ is equivalent to the
usual $\ell_1(\kappa^+)$-basis. For this reason, we consider a
finitely supported vector $x=\sum_{s\in\mathcal{D}}\lambda_se_s\in
X_\kappa$ with $\|x\|\le 1$. We can write $x=x_1+x_2+x_3$, where
$x_1=\sum_{lev(s)\leq \eta_1}\lambda_se_s$, $x_2=\sum_{\eta_1+1\le
lev(s)\leq \eta_2}\lambda_se_s$ and $x_3=\sum_{\eta_2+1\le lev(s)}
\lambda_se_s$. Clearly, $\|x_j\|\leq \|x\|=1$ for any $j=1,2,3$.
Then
$$ \Big|\sum_{i=1}^{2n}a_iB_i^\ast(x)\Big| \le
\Big|\sum_{i=1}^{2n}a_iB_i^\ast(x_1)\Big| +
\Big|\sum_{i=1}^{2n}a_iB_i^\ast(x_2)\Big| +
\Big|\sum_{i=1}^{2n}a_iB_i^\ast(x_3)\Big|.$$ Now we have,

\begin{align*}
    \Big|\sum_{i=1}^{2n}a_iB_i^\ast(x_3)\Big| & \leq
    \Big(\sum_{i=1}^{2n}a_i^2\Big)^{1/2}
    \Big(\sum_{i=1}^{2n}\abs{B_i^\ast(x_3)}^2\Big)^{1/2} \leq \Big(\sum_{i=1}^{2n}a_i^2\Big)^{1/2}\\
    \Big|\sum_{i=1}^{2n}a_iB_i^\ast(x_2)\Big| & \leq
    \Big(\sum_{i=1}^{2n}a_i^2\Big)^{1/2}
    \Big(\sum_{i=1}^{n}\abs{B_i^\ast(x_2)}^2+\sum_{i=n+1}^{2n}\abs{B_i^\ast(x_2)}^2\Big)^{1/2}\\
    &\leq \Big(\sum_{i=1}^{2n}a_i^2\Big)^{1/2} (2\|x_2\|^2)^{1/2} \leq \sqrt{2}\Big(\sum_{i=1}^{2n}a_i^2\Big)^{1/2}
\end{align*}

\begin{align*}
    \Big|\sum_{i=1}^{2n}a_iB_i^\ast(x_1)\Big| & = \Big|\sum_{i=1}^{n}(a_iB_i^\ast(x_1)+a_{n+i}B_{n+i}^\ast(x_1))\Big|\\
    &= \Big|\sum_{i=1}^{n}(a_i+a_{n+i})S_i^\ast(x_1)\Big| \leq \sum_{i=1}^{n}\abs{a_i+a_{n+i}} \abs{S_i^\ast(x_1)}\\
    &\leq \sum_{i=1}^{n}\abs{a_i+a_{n+i}}.
\end{align*}
Summarizing the above, for any finitely supported $x\in X_\kappa$
with $\|x\|\le 1$ we have
$$\Big|\sum_{i=1}^{2n}a_iB_i^\ast(x)\Big| \leq (\sqrt{2}+1)
\Big(\sum_{i=1}^{2n}a_i^2\Big)^{1/2} +
\sum_{i=1}^{n}\abs{a_i+a_{n+i}}.$$
 Therefore, $\|\sum_{i=1}^{2n}a_iB_i^\ast\|\le (\sqrt{2}+1)
(\sum_{i=1}^{2n}a_i^2)^{1/2} + \sum_{i=1}^{n}\abs{a_i+a_{n+i}}$. On
the other hand, $\Delta^\ast$ is equivalent to the usual
$\ell_1(\kappa^+)$-basis. It follows that
$$\delta \sum_{i=1}^{2n}\abs{a_i} \le (\sqrt{2}+1)
\Big(\sum_{i=1}^{2n}a_i^2\Big)^{1/2} +
\sum_{i=1}^{n}\abs{a_i+a_{n+i}}.$$
 If we choose $a_1=\ldots=a_n=1$ and $a_{n+1}=\ldots=a_{2n}=-1$,
then we obtain $\delta\le \frac{\sqrt{2}+1}{\sqrt{2n}}$ for any
$n\in\mathbb{N}$ and we reach a contradiction.
\end{proof}

\section{The non-separable version of Rosenthal's $\ell_1$-theorem}\label{sec.Rosenthal}
In this section, we show that we can not achieve a satisfactory
extension of Rosenthal's $\ell_1$-theorem to spaces of the type
$\ell_1(\kappa)$, for $\kappa$ an uncountable cardinal. As it was
mentioned in the introduction, this extension is possible in only
one case, namely when both $\kappa$ and $cf(\kappa)$ are strong
limit cardinals. For the proof of this result we refer to
\cite{Gryll} and we shall discuss the other cases.

Suppose first that $\kappa$ is not a strong limit cardinal. This
means that there exists a cardinal $\lambda<\kappa$ with $\kappa\le
2^\lambda$. We now consider the space $X_\lambda$ and the
corresponding family of functionals $\Gamma^\ast\subset
X_\lambda^\ast$. Then, $\Gamma^\ast$ is a bounded subset of
$X_\lambda^\ast$ whose cardinality is equal to $2^\lambda\ge\kappa$.
Further, by Corollary \ref{cor.no weakly cauchy}, the set
$\Gamma^\ast$ contains no weakly Cauchy sequence and, by Theorem
\ref{th.basiko}, no subset of $\Gamma^\ast$ is equivalent to the
usual $\ell_1(\kappa)$-basis.

We next consider the case where $\kappa$ is strong limit but
$cf(\kappa)$ is not a strong limit cardinal. This case is not so
simple as the previous one, however it is essentially based on the
arguments developed in Section \ref{sec.Main Results}.

Since $cf(\kappa)$ is not strong limit, there exists a cardinal
$\lambda<cf(\kappa)$ with $cf(\kappa)\le 2^\lambda$. By the
definition of $cf(\kappa)$, there are cardinals $\{\kappa_i\mid
i<cf(\kappa)\}$ such that $\kappa_i<\kappa$, for any ordinal
$i<cf(\kappa)$, and $\kappa=\sum_{i<cf(\kappa)}\kappa_i$. We next
consider the space $X_\kappa$ and we choose a family of branches
$A\subset \Gamma$ as follows. We focus on the level $\lambda$ of the
tree $\mathcal{D}$. This level consists of the nodes
$\{0,1\}^\lambda=\{(a_\xi)_{\xi<\lambda}\mid a_\xi=0 ~\text{or}~
1\}$. Therefore, there are $2^\lambda$ nodes on the level $\lambda$.
Since $cf(\kappa)\le 2^\lambda$, we can choose nodes $\{t_i\mid
i<cf(\kappa)\}$ on the level $\lambda$ with $t_i\neq t_j$ provided
that $i\neq j$. Now we observe that for any $i<cf(\kappa)$, the set
of all branches passing through the node $t_i$ has cardinality
$2^\kappa$. Hence, for any $i<cf(\kappa)$, we can choose a family of
branches $A_i\subset \Gamma$ such that $|A_i|=\kappa_i$ and each
branch belonging to $A_i$ passes through the node $t_i$. Finally,
let $A=\cup_{i<cf(\kappa)} A_i$ and let $A^\ast$ be the family of
the corresponding functionals, that is $A^\ast=\{B^\ast \mid B\in
A\}$.

Clearly, the choice of the family $A$ implies that
$|A^\ast|=|A|=\sum_{i<cf(\kappa)} \kappa_i=\kappa$. Furthermore, by
Corollary \ref{cor.no weakly cauchy}, $A^\ast$ contains no weakly
Cauchy sequence. So, it remains to show that no subset of $A^\ast$
is equivalent to the usual $\ell_1(\kappa)$-basis. The proof follows
the lines of the proof of Theorem \ref{th.basiko}. We describe
briefly the corresponding of Lemma \ref{lemma1}.

\begin{lemma}\label{lemma1a}
Let $\Delta$ be a subset of $A$ with $|\Delta|=\kappa$. Then there
exists a node $s\in\mathcal{D}$ such that $lev(s)<\lambda$ and
$|\Delta_{s\cup\{0\}}|=|\Delta_{s\cup\{1\}}|=\kappa$. (Recall that
$\Delta_s=\{B\in \Delta\mid s\in B\}$.)
\end{lemma}

\begin{proof}
Assuming that the assertion is not true, we construct an initial
segment $S=\{s_\eta\}_{\eta<\lambda}=\{s_0<s_1<\ldots\}$ such that
$|\Delta_{s_\eta}|=\kappa$ for any $\eta<\lambda$. We start with
$s_0=\emptyset$. If $\eta=\eta_0+1$, then $s_\eta$ is one of the
followers of $s_{\eta_0}$. If $\eta$ is a limit ordinal, then we set
$s_\eta=\cup_{\xi<\eta}s_\xi$. Clearly, $s_\eta$ is a node on the
$\eta$-th level of $\mathcal{D}$. We next show that
$$\Delta_{s_\eta}^c=\cup_{\xi<\eta}(\Delta_{s_\xi}\cap \Delta_{s_{\xi+1}}^c).$$
Therefore, $|\Delta_{s_\eta}^c|=\sum_{\xi<\eta}|\Delta_{s_\xi}\cap
\Delta_{s_{\xi+1}}^c|<\kappa$, since $|\Delta_{s_\xi}\cap
\Delta_{s_{\xi+1}}^c|<\kappa$ and $\eta<\lambda<cf(\kappa)$. Hence
$|\Delta_{s_\eta}|=\kappa$ and this completes the construction of
$S$.

Finally, we set $s_\lambda=\cup_{\xi<\lambda}s_\xi$. Then
$s_\lambda$ belongs to the level $\lambda$ and as previously we show
$|\Delta_{s_\lambda}|=\kappa$. However, the choice of $A$ indicates
that $|\Delta_s|<\kappa$ for any node $s$ on the level $\lambda$ and
we have reached a contradiction.
\end{proof}

Using Lemma \ref{lemma1a}, we construct a countable subtree
$T=\{t_1,t_2,t_3,\ldots\}$ of $\mathcal{D}$ such that:
\begin{enumerate}
    \item $|\Delta_{t_m}|=\kappa$ for any $m=1,2,\ldots$ (therefore,
    $lev(t_m)<\lambda$);
    \item the successors $t_{2m},t_{2m+1}$ of the node $t_m$ are the
    successors of some node $s_m\in\mathcal{D}$.
\end{enumerate}
Finally, we repeat the proof of Theorem \ref{th.basiko} to show that
no subset $\Delta^\ast$ of $A^\ast$ is equivalent to the usual
$\ell_1(\kappa)$-basis.

\section{The structure of the subspaces of
$X_\kappa$}\label{sec.structure} The structure of the subspaces of
the James Tree space ($JT$) and the Hagler Tree space ($HT$) has
been studied extensively, since it has provided answers to several
questions about Banach spaces. By analogy, the structure of the
subspaces of $X_\kappa$ seems quite interesting. This section is
devoted to some remarks concerning this issue.

First of all, $X_\kappa$ contains a lot of subspaces isomorphic to
$c_0(\kappa)$. Indeed, let $B=\{s_\eta\}_{\eta<\kappa}$ be any
branch and, for any $\eta<\kappa$, let $t_\eta$ be the successor of
$s_\eta$ with $t_\eta\neq s_{\eta+1}$. Then $\{t_\eta \mid
\eta<\kappa\}$ is a strongly incomparable family of nodes. By
Proposition \ref{prop.c0}, it follows that the subspace
$\overline{span}\{e_{t_\eta}\mid \eta<\kappa\}$ is isomorphic to
$c_0(\kappa)$. Furthermore, it is easy to verify that for any
ordinal $\eta<\kappa$ the subspace $\overline{span}\{e_s\mid
s\in\{0,1\}^\eta\}$ is isometrically isomorphic to the space
$\ell_2(2^\eta)$. The main properties of the spaces $JT$ and $HT$
suggest now the following problem about the subspaces of $X_\kappa$.

\begin{problem}
Is it true that there exists no subspace of $X_\kappa$ isomorphic to
$\ell_1(\kappa)$?
\end{problem}

Concerning the above problem, we prove a partial result. Assume that
$B=\{s_\eta\}_{\eta<\kappa}$ is any branch of the tree
$\mathcal{D}$. Then we show that the subspace generated by this
branch, that is the subspace
$\overline{span}\{e_{s_\eta}\}_{\eta<\kappa}$, does not contain any
copy of $\ell_1(\kappa)$.

For our convenience, we first define a Banach space isometrically
isomorphic to the subspace generated by any branch. Let $\kappa$ be
an infinite cardinal. We consider the vector space
$c_{00}(\{\eta\mid\eta<\kappa\})$ consisting of all finitely
supported functions $x:\{\eta \mid \eta<\kappa\}\to\mathbb{R}$. For
any $x\in c_{00}(\{\eta\mid\eta<\kappa\})$, we set
$$\|x\|=\sup\{\abs{S^\ast(x)}\}$$
where the supremum is taken over all segments
$S\subseteq\{\eta\mid\eta<\kappa\}$. If $E_\kappa$ denotes the
completion of the normed space we have just defined, then $E_\kappa$
is isometrically isomorphic to the subspace of $X_\kappa$ generated
by any branch.

As usual, for any ordinal $\eta<\kappa$, we consider the vector
$e_\eta\in E_\kappa$ with $e_\eta(\xi)=1$ if $\xi=\eta$ and
$e_\eta(\xi)=0$ otherwise. We now define some projections on the
space $E_\kappa$. Let $\eta$ be any ordinal, $\eta<\kappa$. We
define $P_\eta:span\{e_\xi\}_{\xi<\kappa}\to
span\{e_\xi\}_{\xi<\eta}$ as follows: if
$x=\sum_{\xi<\kappa}x(\xi)e_\xi$ is finitely supported, then
$P_\eta(x)=\sum_{\xi<\eta}x(\xi)e_\xi$. Clearly, $P_\eta$ is a
linear projection with $\|P_\eta\|=1$. We can also extend $P_\eta$
continuously and we obtain a projection $P_\eta:E_\kappa\to
E_\kappa$ onto $\overline{span}\{e_\xi\}_{\xi<\eta}$ with
$\|P_\eta\|=1$. We next prove the following.

\begin{proposition}\label{prop.sbspa}
The space $E_\kappa$ does not contain any isomorphic copy of
$\ell_1(\kappa)$.
\end{proposition}

\begin{proof}
Suppose, on the contrary, that $\ell_1(\kappa)$ embeds
isomorphically into $E_\kappa$. Then we find a subset $\{x_\xi\mid
\xi<\kappa\}$ of $E_\kappa$ which is equivalent to the usual
$\ell_1(\kappa)$-basis. Without loss of generality, we may assume
that $x_\xi$ is finitely supported and $\|x_\xi\|=1$ for any
$\xi<\kappa$.

We inductively construct a sequence $(y_m)_{m=0}^\infty$ belonging
to $span\{e_\xi\}_{\xi<\kappa}$ with the following properties:
\begin{enumerate}
    \item $\|y_m\|=1$ for each $m$;
    \item if $A_m\subset\{\xi<\kappa\}$ is the support of $y_m$ then
    $\max A_m<\min A_{m+1}$ for any $m$;
    \item $(y_m)_{m=0}^\infty$ is a block sequence of
    $(x_\xi)_{\xi<\kappa}$, that is there are ordinals
    $\eta_0<\eta_1<\ldots$ so that $y_m\in span\{x_\xi \mid
    \eta_m\le \xi < \eta_{m+1}\}$.
\end{enumerate}
We start with $y_0=x_0$, $\eta_0=0$ and $\eta_1=1$. Let $\xi_1=\max
A_0+1$. We claim that there exists $y\in span\{x_\xi\}_{\xi\geq 1}$,
$y\neq 0$, such that $P_{\xi_1}(y)=0$. Indeed, if we assume that
$P_{\xi_1}(y)\neq 0$ for all $y\in span\{x_\xi\}_{\xi\geq 1}$,
$y\neq 0$, then the linear operator $P_{\xi_1}:
span\{x_\xi\}_{\xi\geq 1}\to span\{e_\xi\}_{\xi<\xi_1}$ is
one-to-one. Since $\{x_\xi\}_{\xi\geq 1}$ are linearly independent,
it follows that the (algebraic) dimension of the vector space
$span\{e_\xi\}_{\xi<\xi_1}$ is equal to $\kappa$, which is a
contradiction. Therefore, there is $y\in span\{x_\xi\}_{\xi\geq 1}$
such that $y\neq 0$ and $P_{\xi_1}(y)=0$. We set $y_1=y/\|y\|$.
Since $P_{\xi_1}(y)=0$, we have $\max A_0<\min A_1$. Moreover, we
can choose an ordinal $\eta_2>\eta_1$ such that $y\in span\{x_\xi
\mid \eta_1\leq \xi <\eta_2\}$. Applying repeatedly the previous
argument, we construct the desired sequence $(y_m)_{m=0}^\infty$.

Since $(x_\xi)_{\xi<\kappa}$ is equivalent to the usual
$\ell_1(\kappa)$-basis, it is easy to verify that the sequence
$(y_m)$ is equivalent to the usual $\ell_1$-basis. Furthermore, the
sequence $(y_m)$ belongs to $span\{e_\xi \mid \xi \in
\cup_{m=0}^\infty A_m\}$. The latter space is isometrically
isomorphic to $E_{\aleph_0}$, which in turn is isomorphic to $c_0$
(see \cite{Hagler}). That is, in a space isomorphic to $c_0$ we find
a copy of $\ell_1$, which is a contradiction.
\end{proof}

\bibliographystyle{amsplain}

\end{document}